\documentclass[12pt]{article} 

\usepackage{amsmath}
\usepackage{amsthm}
\usepackage{amsfonts}
\usepackage{mathrsfs}
\usepackage{stmaryrd}
\usepackage{setspace}
\usepackage{fullpage}
\usepackage{amssymb}
\usepackage{breqn}
\usepackage{enumitem}
\usepackage{bbold} 
\usepackage{authblk}
\usepackage{comment}
\usepackage{hyperref}
\usepackage{pgf,tikz}
\usepackage{graphicx}
\usetikzlibrary{decorations.pathreplacing,arrows}
\tikzstyle{vertex}=[circle,draw=black,fill=black,inner sep=0,minimum size=3pt,text=white,font=\footnotesize]

\newtheorem{thm}{Theorem}[section]

\newtheorem{lemma}[thm]{Lemma}

\newtheorem{clm}[thm]{Claim}

\newtheorem*{lemma*}{Lemma}
\newtheorem*{proposition*}{Proposition}
\newtheorem*{theorem*}{Theorem}

\newcommand\ex{\ensuremath{\mathrm{ex}}}

\newcommand\cH{{\mathcal H}}

\newcommand\cN{{\mathcal N}}

\newcommand{\ignore}[1]{}

\title{Degree powers and number of stars in graphs with a forbidden broom}

\author{Dániel Gerbner\\ \small Alfr\'ed R\'enyi Institute of Mathematics, HUN-REN\\
\small \texttt{gerbner.daniel@renyi.hu}}
\date{}

\begin{document}

\maketitle

\begin{abstract} 
 Given a graph $G$ with degree sequence $d_1,\dots, d_n$ and a positive integer $r$, let $e_r(G)=\sum_{i=1}^n d_i^r$. We denote by $\ex_r(n,F)$ the largest value of $e_r(G)$ among $n$-vertex $F$-free graphs $G$, and by $\ex(n,S_r,G)$ the largest number of stars $S_r$ in $n$-vertex $F$-free graphs. The \textit{broom} $B(\ell,s)$ is the graph obtained from an $\ell$-vertex path by adding $s$ new leaves connected to a penultimate vertex $v$ of the path. 
 
 We determine $\ex_r(n,B(\ell,s))$ for $r\ge 2$, any $\ell,s$ and sufficiently large $n$, proving a conjecture of Lan, Liu, Qin and Shi. We also determine $\ex(n,S_r,B(\ell,s))$ for $r\ge 2$, any $\ell,s$ and sufficiently large $n$.
\end{abstract}

\section{Introduction}

Given a graph $G$ with degree sequence $d_1,\dots, d_n$ and a positive integer $r$, let $e_r(G)=\sum_{i=1}^n d_i^r$. This is a well-studied graph parameter, see the survey \cite{agmm}. Caro and Yuster \cite{cy} initiated the study of $\ex_r(n,F):=\max\{e_r(G): \text{ $G$ is an $n$-vertex $F$-free graph}\}$.

A closely related topic is counting stars in $F$-free graphs. More generally, let $\cN(H,G)$ denote the number of copies of $H$ in $G$ and $\ex(n,H,F):=\max\{\cN(H,G):  \text{ $G$ is an $n$-vertex}$ $F$-free graph$\}$. Clearly for $r>1$ we have $\cN(S_r(G))=\sum_{i=1}^n \binom{d_i}{r}$. The connection of these two topics was investigated by the author \cite{ge5}. We will use only the following simple observation: $\ex_r(n,F)=(r!+o(1))\ex(n,S_r,F)$ for any graph $F$.

In this paper we deal with forbidden brooms. The \textit{broom} $B(\ell,s)$ is the graph obtained from an $\ell$-vertex path $P_\ell$ by adding $s$ new leaves connected to a penultimate vertex $v$ of the path. We call $v$ the \textit{center} of the broom. Caro and Yuster \cite{cy} observed that $\ex(n,B(4,s))=e_r(S_n)$ for $r\ge 2$ and $n>2s+4$. Let $H(k,n)$ denote the $n$-vertex graph with $k$ vertices of degree $n-1$ and $n-k$ vertices of degree $k$. In other words, we take $K_{k,n-k}$ and add all the edges inside the part of order $k$.
Let $H^*(k,n)$ be the graph obtained from $H(k,n)$ by adding an arbitrary edge. Let $F_n$ be the graph obtained from $S_n$ by adding $\lfloor (n-1)/2\rfloor$ independent edges between the leaves.

Lan, Liu, Qin and Shi \cite{llqs} stated a conjecture on $\ex_r(n,B(\ell,s))$ and proved the conjecture for $5\le \ell\le 7$. Wang and Yin \cite{wy} proved it for $\ell=8$. Here we prove the conjecture in full, and also the analogous result on $\ex(n,S_r,B(\ell,s))$. Let $k=\lfloor (l-2)/2\rfloor$.

\begin{thm}\label{main} \textbf{(i)}
  If $\ell$ is even, $r\ge 2$ and $n$ is sufficiently large, then $\ex_r(n,B(\ell,s))=e_r(H(k,n))$. If $\ell$ is odd and $\ell\ge 7$ or $\ell=5$ and $s=0$, then $\ex_r(n,B(\ell,s))=e_r(H^*(k,n))$. If $\ell=5$ and $s>0$, then $\ex_r(n,B(\ell,s))=e_r(F_n)$. Moreover, $H(k,n)$, $H^*(k,n)$ and $F_n$ are the unique extremal graphs in each case.

  \textbf{(ii)} If $\ell$ is even, $r\ge 2$ and $n$ is sufficiently large, then $\ex(n,S_r,B(\ell,s))=\cN(S_r,H(k,n))$. If $\ell$ is odd and $\ell\ge 7$ or $\ell=5$ and $s=0$, then $\ex(n,S_r,B(\ell,s))=\cN(S_r,H^*(k,n))$. If $\ell=5$ and $s>0$, then $\ex(n,S_r,B(\ell,s))=\cN(S_r,F_n)$. Moreover, $H(k,n)$ and $H^*(k,n)$ are the unique extremal graphs in the respective cases.
\end{thm}

We remark that both \cite{llqs} and \cite{wy} obtained reasonably small thresholds on $n$. We do not calculate the threshold, nor make any attempt to optimize it, but it seems likely that our methods would give a very large threshold. 

The case $r=1$ of both $\ex_r(n,B(\ell,s))$ and $\ex(n,S_r,B(\ell,s))$ correspond to the ordinary Tur\'an number $\ex(n,B(\ell,s))$. It was determined for $\ell=4$ and $\ell=5$ in \cite{suwa}. We will only use the simple and well-known bound $\ex(n,B(\ell,s))\le (\ell+s)n$, see e.g. \cite{stein} for a slightly stronger bound that holds for every tree. 

\section{Proof of Theorem \ref{main}}\label{three}

    The lower bounds are obvious. We will prove the upper bounds together. Let $G$ be an extremal graph for $\ex_r(n,B(\ell,s))$ or $\ex(n,S_r,B(\ell,s))$. We will show that $G$ contains $H(k,n)$. We will choose $\gamma,\delta,\varepsilon >0$ such that each of these numbers is small enough with respect to the previous one and to $\ell,s,r$.
    
    We will use only the following corollary of the $B(\ell,s)$-free property: there is no vertex of degree at least $\ell+s$ that is an endpoint of a path on $\ell-1$ vertices in $G$.


    
We look at the $r$-sets in $V(G)$. We will be interested in the number of common neighbors of those $r$ vertices. Let $\cH_1$ denote the $r$-uniform hypergraph that has as hyperedges the  $r$-sets that have more than $k$ common neighbors. Let us fix $\varepsilon>0$ and assume that $n$ is large enough.

 \begin{clm}\label{1.7}
     There are at most $\varepsilon n^r$ hyperedges in $\cH_1$.
 \end{clm}

 \begin{proof}[Proof of Claim]
     Let us delete vertices that are contained in at most $\varepsilon n^{r-1}$ hyperedges in $\cH_1$. Then we repeat this for the resulting hypergraph, and so on. After some steps, we either have deleted all the vertices (in which case we deleted $n$ times at most $\varepsilon n^{r-1}$ hyperedges and we are done), or we arrived to an $m$-vertex hypergraph $\cH_1'$ where each vertex is in more than $\varepsilon n^{r-1}$ hyperedges.


Let $d=\lfloor c_0\varepsilon^{1/(r-1)}n \rfloor$ and let $A$ be the set of vertices in $G$ with degree more than $d$, where we pick $c_0>0$ sufficiently small. Recall that 
the $B(\ell,s)$-free property implies that $G$ has at most $(\ell+s)n$ edges. This implies that $|A|\le 2(s+\ell)/c_0\varepsilon^{1/(r-1)}$.

     Let $B=V(\cH_1')$ and $C$ be the set of vertices that are common neighbors of at least one hyperedge of $\cH_1'$. Let us consider a shortest path $P$ between a vertex $v$ of $G$ with degree at least $\ell+s$ and a vertex $u_1\in B\cup C$. Note that this might be a path of length 0 if there is a high degree vertex in $B\cup C$. Let $N(u)$ denote the set of vertices that are contained in a hyperedge of $\cH_1'$ that also contains $u$. If $u\in B$, then there are at least $\varepsilon n^{r-1}$ hyperedges containing $u$, thus there are at least $\varepsilon' n$ vertices in $N(u)$ for some constant $\varepsilon'>0$. We let $N'(u)=N(u)\setminus A$, then $|N'(u)|\ge \varepsilon'' n$ for some constant $\varepsilon''>0$.
     
     We build a path $P'$ in $G$ gradually. We start with $P$. Then we pick a vertex in $N'(u_1)$ to be the vertex $u_3$ that comes after the next vertex in the path. We pick a vertex from $N'(u_1)$ that is not on the path so far, this can be done greedily. We will later give more details on how we pick $u_3$. We pick a common neighbor of $u_1$ and $u_3$ as $u_2$. We have at least $k+1$ choices by the definition of $\cH_1$. We want to pick one that is not on the path $P'$ yet. We will continue this till we find $u_{2k+1}$. 

     If $u_1=v$, we look at the vertices that are common neighbors of a hyperedge containing $u_1$. If there is one not in $A$, we pick that as $u_2$ and we pick as $u_3$ an arbitrary vertex that is in a hyperedge with common neighbor $u_2$, containing $u_1$.
     
     Let us consider now the step when we have found $u_{2i-1}$ and we want to pick $u_{2i+1}$ and $u_{2i}$. We look at each vertex $u \in N'(u_{2i-1})$ and the common neighbors of $u_{2i-1}$ and $u$. Note that the vertices of $P$ are not in $B\cup C$, thus cannot be picked as $u_{2i}$ or $u_{2i+1}$.  Each vertex $u_{2j+1}$ with $j\ge 1$ is not in $A$. Note that $u_{2j+1}$ may be a common neighbor of a hyperedge in $\cH_1'$ that contains $u_{2i-1}$, but there are at most $\binom{c_0\varepsilon^{1/(r-1)}n}{r-1}\le c_1\varepsilon n^{r-1}$ such hyperedges. The same holds for $u_2$ if $u_2\not\in A$.
     
     Recall that there are at least $\varepsilon n^{r-1}$ hyperedges in $\cH_1'$ that contain $u_{2i-1}$. If $c_0$ is sufficiently small, then we have at least $(\ell+s-1)\varepsilon n^{r-1}/(\ell+s)$ hyperedges in $\cH_1'$ that contain $u_{2i-1}$ and the common neighbor is not $u_{2j+1}$. Applying this for every $j$ and also for $u_2$ if $u_2\not\in A$, we obtain at least $c_2\varepsilon n^{r-1}$ hyperedges that contain $u_{2i-1}$ and the common neighborhood is disjoint from $u_3,\dots,u_{2i-3}$, and also from $u_2$ if $u_2\not\in A$, for some constant $c_2>0$. Therefore, we have to pick a common neighbor that avoids $u_1$ and $u_2,\dots,u_{2i-2}$. This can be done greedily as long as $i\le k$. If $u_2\not\in A$, then this can be done for $i=k+1$ as well.
     
     This way we found a path $u_1u_2\dots u_{2k}u_{2k+1}$ in $G$ that shares only $u_1$ with $P$, thus can be extended to a path with first vertex $v$. If this path has at least $\ell-1$ vertices, we obtain a contradiction. Otherwise $\ell$ is odd and $v=u_1$. If $u_2\not\in A$, we can continue the path $P'$ to $u_{2k+3}$ and again $v$ is the endvertex of a path of length at least $\ell-1$.
     

Finally, assume that the common neighborhood of every hyperedge containing $u_1$ consists of vertices in $A$. There are at least $k+1$ such vertices, and none of them is among $u_3,\dots,u_{2k+1}$. Then at most $k$ of them are on $P'$, thus we can add one of them as $u_0$ to $P'$. This way we found a path of length $\ell-1$ with first vertex $u_0\in A$, thus of degree more than $\ell+s$, a contradiction.
\end{proof}

Let us return to the proof of the Theorem. Let $\cH_2$ denote the subhypergraph of $\cH_1$ where an $r$-set is a hyperedge of $\cH_2$ if its vertices have more than $\ell+s$ common neighbors.

\begin{clm}
   There is a constant $c=c(\ell,s)$ such that $\cH_2$ has at most $cn$ hyperedges.
\end{clm}

\begin{proof}[Proof of Claim]
    We show that $\cH_2$ does not contain a Berge path of length $k+1$. Such a subhypergraph would consist of $k+1$ hyperedges $h_1,\dots,h_{k+1}$ and $k+2$ vertices $v_1,\dots,v_{k+2}$ such that $h_i$ contains $v_i$ and $v_{i+1}$ for every $i\le k+1$. If such a Berge path exists in $\cH_2$, then for each $h_i$ we can pick a vertex $u_i$ that is a common neighbor of the vertices in $h_i$ and is different from the vertices of the Berge path and the vertices $u_j$ picked earlier. This is doable since $u_i$ has to be distinct from at most $2k+2\le \ell+s$ vertices. This way we found a path $v_1u_1v_2u_2\dots v_{k+1}u_{k+1}v_{k+2}$ of length at least $\ell-1$ with endvertices of degree more than $\ell+s$, a contradiction.

    It is well-known that an $r$-uniform hypergraph without a Berge path of a given length has at most $cn$ hyperedges, see \cite{gykl,dgymt} for exact results.
\end{proof}

Let us return to the proof of the Theorem. Let $\cH_3$ (resp. $\cH_4$) denote the hypergraph which has as hyperedges the $r$-sets with common neighborhood of order exactly $k$ (resp. less than $k$). 

\begin{clm}
    $\cH_4$ has at most $\delta\binom{n}{r}$ hyperedges.
\end{clm}

\begin{proof}[Proof of Claim]
    Let us count the copies of $S_r$ by picking an $r$-set, i.e., a hyperedge as the set of leaves, and a vertex in the common neighborhood as the center. The hyperedges in $\cH_2$ contribute at most $cn^2$, since we can pick the hyperedge at most $cn$ ways and the center at most $n$ ways. The other hyperedges in $\cH_1$ contribute at most $\varepsilon (\ell+s) n^r$. The hyperedges in $\cH_4$ contribute at most $(k-1)|E(\cH_4)|$, while the hyperedges in $\cH_3$ contribute at most $k|E(\cH_3)|$. This means that $\cN(S_r,G)\le cn^2+\varepsilon (\ell+s) n^r +k\binom{n}{r}-|E(\cH_4)|$. If $r>2$ and $\varepsilon$ is sufficiently small, then $cn^2+\varepsilon (\ell+s) n^r<\delta \binom{n}{r}-k\binom{n-1}{r-1}$, thus either $|E(\cH_4)|\le \delta\binom{n}{r}$, or $\cN(S_r,G)<k\binom{n}{r}-k\binom{n-1}{r-1}=k\binom{n-1}{r}\le\cN(S_r,H(k,n))$, a contradiction.

    In the case $r=2$ our upper bound $cn^2$ on the contribution of $\cH_2$ is not negligible. However, we can obtain the upper bound $\varepsilon n^2$ if $\cH_2$ has at most $\varepsilon n$ hyperedges or if all but $\varepsilon n/2$ hyperedges in $\cH_2$ have at most $\varepsilon n/2c$ vertices in the common neighborhood. If none of these assumptions hold, then we have at least $\varepsilon n/2$ pairs in $G$ that have at least $\varepsilon n/2c$ vertices in the common neighborhood.
 At least $\Omega(\sqrt{n})$ vertices belong to at least one such pair, and each of those vertices have linear degree in $G$. Therefore, $G$ has $\Omega(n^{3/2})$ edges, a contradiction to $|E(G)|\le \ex(n,B(\ell,s))=O(n)$.

    It is easy to see that the same argument works if we deal with $e_r(G)$. Indeed, we obtain $e_r(G)\le (r!+o(1))\left( cn^2+\varepsilon (\ell+s) n^r +k\binom{n}{r}-|E(\cH_4)|\right)$. Here we use the better bound $cn^2+\varepsilon (\ell+s) n^r<\delta \binom{n}{r}/2-k\binom{n-1}{r-1}$.
    Then we have either $|E(\cH_4)|\le \delta\binom{n}{r}$, or $e_r(G)<(r!+o(1))\left(k\binom{n-1}{r}-\delta\binom{n}{r}/2\right)\le e_r(H(k,n))$, a contradiction.
\end{proof}

We obtained that there are at least $(1-\varepsilon-\delta)\binom{n}{r}$ hyperedges in $\cH_3$. 

\begin{clm}
    There is a $k$-set $B$ such that there are at least $(1-\gamma)\binom{n}{r}$ $r$-sets with common neighborhood $B$.
\end{clm}

\begin{proof}[Proof of Claim]
    Let us first delete the hyperedges that have a vertex of degree less than $\ell+s$ in their common neighborhood. Such a vertex is in the common neighborhood of at most $\binom{\ell+s}{r}$ hyperedges, thus we deleted at most $\binom{\ell+s}{r}n$ hyperedges altogether, let $\cH_3'$ be the resulting hypergraph.

Let us now delete every vertex that is in less than $\varepsilon n^{r-1}$ hyperedges of $\cH_3'$. We repeat this until we obtain $U$ where every vertex is in at least  $\varepsilon n^{r-1}$
hyperedges. Clearly $U$ is not empty, since we delete at most $\varepsilon n^r$ hyperedges. 

We claim that $U$ contains a connected component of order at least $n-\gamma n/r^2$. Indeed, 
if there is no connected component in $U$ of order at least $n-\gamma n/r^2$, then the hyperedges intersecting the largest component and some other component are missing, and there are at least $c'\gamma{^2} n^r$ such hyperedges for some constant $c'$. This is impossible if $\varepsilon$ and $\delta$ are small enough, thus there is a component with vertex set $U'$ of order at least $n-\gamma n/r^2$. 

We will build a path similarly to the proof of Claim \ref{1.7}. Recall that $A$ denotes the set of vertices with degree more than $d=\lfloor c_0\varepsilon^{1/(r-1)}n \rfloor$ in $G$. Let $N''(u)$ is the set of vertices not in $A$ that are contained in a hyperedge of $\cH_3'$ together with $u$. Note that this is analogous to the definition of $N'(u)$, but $\cH_1'$ is replaced by $\cH_3'$. We have that $|N''(u)|\ge \varepsilon'' n$, analogously to the proof of Claim \ref{1.7}.

We pick a vertex $u_1\not\in A$. Let us consider the step when we have found $u_{2i-1}$ and want to pick the next two vertices. We look at each vertex $u\in N''(u_{2i-1})$ and the common neighbors of a hyperedge of $\cH_3'$ containing $u$ and $u_{2i-1}$. As in Claim \ref{1.7}, we eventually obtain that we can pick a hyperedge $h_{i}$ of $\cH_3'$ such that the common neighborhood of $h_{i}$ is in $A$. Thus we have to pick a vertex from the common neighborhood that is not $u_2,u_4,\dots,u_{2i-2}$. This is doable as long as $i\le k$. This way we found $u_{2k+1}$. 

Assume first that there is a hyperedge $h_{k+1}$ in $\cH_3'$ that contains $u_{2k+1}$ and a vertex $u_{2k+3}$ that has not been picked earlier and the common neighborhood of $h_{2k+1}$ is distinct from the common neighborhood of $h_j$ for some $j\le 2k-1$. Then we have $x\in h_{j}\setminus h_{k+1}$. If $x$ was never picked, we change the earlier pick and pick $x$ for $h_j$. This way there is an unused vertex in the common neighborhood of $h_{k+1}$ that has not been used. We pick that and we obtain a path of length at least $\ell-1$ in $G$ with endpoint $u_2$, that is the common neighbor of a hyperedge in $\cH_3'$, hence $u_2$ has degree at least $\ell+s$, a contradiction.

Assume now that there is no such hyperedge $h_{k+1}$. In particular this implies that each $h_i$ with $i\le k$ has the same common neighborhood $B$. Let us now delete each hyperedge that contains a vertex of $B$ and let $\cH_3''$ be the resulting hypergraph. Clearly we deleted at most $k\binom{n}{r-1}$ hyperedges. Assume that there is a hyperedge $h$ of $\cH_3''$ with common neighborhood $C$ such that $C\neq B$. Then there is a Berge path starting with $h_1$ and ending in a vertex of $h$. On that Berge path there are two consecutive hyperedges such that one of them has common neighborhood $B$, while the other not. We can assume without loss of generality that the second one is $h$, and it shares a vertex $u$ with a hyperedge $h'$ of $\cH_3'$ such that the common neighborhood of $h'$ is $B$. 

There is a vertex $u'$ of $h$ that is not among the vertices $u_1,u_3,\dots,u_{2k+1}$, otherwise $B$ would be contained in the common neighborhood of $h$. It is possible that $u=u_{2i+1}$ for some $i$. Consider the path $u_1u_2\dots u_{2k}$. If $u$ is among these vertices, then we replace $u$ with $u_{2k+1}$. This is doable since $u_{2k+1}$ in $B$, which is the common neighborhood of the hyperedge containing $u_{2i}$ and $u_{2i+2}$, thus $u_{2k+1}$ is adjavcent to both $u_{2i}$ and $u_{2i+2}$. After that, we add $u$, $x$ and $u'$ to the end of this path, where $x\in C\setminus B$. This way we obtained a path of length $\ell-1$ with an endvertex of degree at least $\ell+s$, a contradiction.

We obtained that $B$ is the common neighborhood of every hyperedge in $\cH_3''$. Recall that there are at least $(1-\varepsilon-\delta)\binom{n}{r}$ hyperedges in $\cH_3$. We obtained $\cH_3''$ by deleting at most $\binom{\ell+s}{r}n+\varepsilon n^r+\gamma n\binom{n}{r-1}/r^2+k\binom{n}{r-1}$ hyperedges, thus we are done if $\varepsilon$ and $\delta$ are small enough.
\end{proof}

Let us return to the proof of the Theorem. Clearly there are at least $(1-\gamma)n$ vertices in the common neighborhood $W$ of $B$, thus we obtained a copy of $K_{k,n-\gamma n}$ in $G$.We claim that any vertex $v$ outside $B$ has degree at most $\gamma^{1/r}n$. Indeed, no $r$-sets in the neighborhood of $v$ belongs to $\cH_3''$, thus $\binom{d(v)}{r}\le \gamma\binom{n}{r}$.

Assume that a vertex $v'$ has degree at most $\ell+s$ and is outside $B$. We claim that $v'$ is adjacent to every vertex of $B$. Indeed, otherwise we delete each edge incident to $v'$ and add the edges between $v'$ and $B$. The resulting graph $G'$ is $B(\ell,s)$-free, since $v'$ could be replaced by any common neighbor of $B$ in a copy of $B(\ell,s)$. We erased at most $\binom{\ell+s}{r}$ copies of $S_r$ with center $v'$ and at most $(\ell+s)\binom{\gamma^{1/r}n}{r-1}$ copies of $S_r$ where the center is a neighbor of $v'$. We added $\binom{n-\gamma n}{r-1}$ new copies of $S_r$, where the center is the new neighbor of $v'$ in $B$. Therefore, $G'$ contains more copies of $S_r$ than $G$. Analogous calculation shows that $e_r(G')>e_r(G)$, a contradiction.

Let $G_0$ denote the connected component of $G$ containing $B\cup W$ and $m=|V(G_0)|$.

\begin{clm}\label{struc} If $\ell$ is even, then
    $G_0$ is a subgraph of $H(k,m)$. If $\ell\ge 7$ is odd, then $G_0$ is a subgraph of $H^*(k,m)$. If $\ell=5$, then $G_0$ is a subgraph of $F_m$. If $\ell=5$ and $s=0$, then $G_0$ is a subgraph of $H^*(k,m)$.
\end{clm}

\begin{proof}[Proof of Claim]
    Observe that the statement is equivalent to the following. If $\ell$ is even, then there is no edge outside $B$, if $\ell\ge 7$ is odd or $\ell=5$ and $s=0$, then there is at most one edge outside $B$, and if $\ell=5$ and $s>0$, then there is a matching outside $B$.

 Consider first the case $\ell$ is even and assume that there is an edge $uv$ outside $B$. Then there is a shortest path from this edge to $B$, and then we can continue this path by alternating between $B$ and $W$. This is a path on at least $2k+2=\ell$ vertices with the penultimate vertex in $B$, thus of degree at least $\ell+s$, a contradiction.

Consider now the case $\ell$ is odd. Assume that there are two edges outside $B$ of the form $uv$ and $uv'$. Then we consider a shortest path $P$ between these vertices and $B$. If $P$ starts at $v$, we can extend $P$ in one direction by $u$ and $v'$, and in the other direction by an alternating path between $B$ and $W$. This path has $2k+3$ vertices and the penultimate vertex is in $B$, a contradiction.
The case $P$ starts at $v'$ is analogous. If $P$ starts at $u$, assume first that the degree of $v$ is more than $\ell+s$. Then we can pick a neighbor $u'$ of $v$ that is not in $B$ and extend $P$ with $v$ and $u'$ in one direction and an alternating path between $B$ and $W$ in the other direction, to obtain a contradiction like before. Assume now that the degree of $v$ is less than $\ell+s$. Then $v$ is adjacent to every vertex of $B$, thus we can start a shortest path to $B$ from $v$, arriving to a case we have already dealt with. Observe that we completed the proof in the case $\ell=5$ and $s>1$.

Finally, assume that there are two independent edges $uv$ and $u'v'$ outside $B$. If any of these four vertices has degree more than $\ell+s$, then we can pick a neighbor that is also outside $B$ and arrive to the previous case. Otherwise these four vertices are each adjacent to each vertex in $B$. Let $x,y$ be vertices in $B$ (if $|B|=k\ge 2$), then we can consider the path $xuvyu'v'$. If $k\ge 3$, we can go to a third vertex of $B$ and then alternate between $B$ and $W$. This way for $k\ge 2$ we obtain a path on $2k+2$ vertices, with an endpoint in $B$, a contradiction. For $k=1$, $s=0$, the path $uvxu'v'$ gives the contradiction.
\end{proof}

Let us return to the proof. Let $Q$ be the set of the vertices not in $G_0$.  We have two cases. If $|Q|\ge\sqrt{n}$, then for large enough $n$ we have that $|Q|$ is sufficiently large, thus all the previous structural statements, in particular Claim \ref{struc} hold for $G[Q]$ as well, thus $G[Q]$ also contains a large component that is a subgraph of $H(k,m')$ or $H^*(k,m')$ or $F_m'$. It is easy to see that combining those components into one $H(k,m+m')$ or $H^*(k,m+m')$ or $F_{m+m'}$ increases both $e_r(G)$ and the number of copies of $S_r$.

The other case is that $|Q|<\sqrt{n}$. Then there are at most $\binom{|Q|}{r+1}$ copies of $S_r$ inside $Q$. Deleting all the edges inside $Q$ and adding all the edges between $Q$ and $B$ does not create $B(\ell,s)$. We created at least $k|Q| \binom{n-\gamma n}{r-1}$ copies of $S_r$. Then the number of copies of $S_r$ increased. It is easy to see that analogously $e_r$ increases. We obtained a contradiction that completes the proof of the upper bound. Moreover, we showed that any extremal graph is a subgraph of $H(k,n)$, $H^*(k,n)$ or $F_n$ in the appropriate cases. It is easy to see that erasing an edge from any of these graphs decreases $e_r(G)$, and erasing an edge from $H(k,n)$ or $H^*(k,n)$ decreases the number of copies of $S_r$, proving the uniqueness where it is claimed.

\vskip 0.3truecm

\textbf{Funding}: Research supported by the National Research, Development and Innovation Office - NKFIH under the grants SNN 129364, FK 132060, and KKP-133819.

\vskip 0.3truecm


\end{document}